\documentclass[a4paper, 11pt]{article}
\usepackage{amsmath}
\usepackage{amssymb}
\usepackage{graphicx}
\newcommand{\re}[1]{(\ref{#1})}
\usepackage{amsmath,amsxtra,amssymb,latexsym, amscd,amsthm}
\usepackage{geometry}
\usepackage[dvipsnames]{xcolor}

\def\R{{\mathbb{R}}}

      \newtheorem{thm}{Theorem}[section]
      \newtheorem{propo}{Proposition}[section]
      
      \newtheorem{Def}{Definition}[section]
      \newtheorem{rmq}{Remark}[section]
      \newtheorem{lem}{Lemma}[section]
      
      \newtheorem{cor}{Corollary}[section]

      \newcommand{\F}{\mathcal{F}}
      \newcommand{\C}{\mathcal{C}}
 \newcommand{\rn}{\mathbb{R}^n}
  \newcommand{\hn}{\mathbb{R}^n_+}
 \newcommand{\co}{\mathcal{C}_\omega}
   \newcommand{\Mp}{\mathcal{M}^+}
      \newcommand{\Mm}{\mathcal{M}^-}

      \newcommand{\ofq}{\overline{F}_{Q,y}}

\begin{document}

\begin{center}{\bf\Large Stationary states of reaction-diffusion and Schr\"odinger systems with inhomogeneous or controlled diffusion  }
\end{center}

 \begin{center}
Boyan SIRAKOV
\\{\small\it PUC-Rio, Departamento de Matem\'atica, G\'avea, Rio de Janeiro, CEP 22451-900, Brazil}\smallskip

Alexandre MONTARU\\
{\small\it  Laboratoire de Math\'ematiques,
Universit\'{e} de Franche-Comt\'{e}, 25030 Besan\c{c}on, France }\end{center}

{\small \noindent{\bf Abstract}.
We obtain  classification, solvability and nonexistence theorems for positive stationary states of reaction-diffusion and Schr\"odinger systems involving a balance between repulsive and attractive terms. This class of systems contains PDE arising in biological models of Lotka-Volterra type, in physical models of Bose-Einstein condensates and in models of chemical reactions. We show, with different proofs, that the results obtained in [ARMA,  213 (2014), 129-169] for models with homogeneous diffusion are valid for general heterogeneous media, and even for controlled inhomogeneous diffusions.}
%

   \section{Introduction}

This paper is a sequel to our recent work \cite{MSS}, in which we studied stationary states of systems of reaction-diffusion PDEs or standing waves of coupled Schr\"odinger systems, including as a particular case the important for applications system
\begin{equation}
\label{model}
\left\{\ {\alignedat2
 -\Delta u  &=u\Big[v^p\left(a(x)v^q-c(x)u^q\right) + \mu(x)\big] \qquad &\mbox{ in }&\Omega, \\
  -\Delta v &=v\Big[u^p\left(b(x)u^q-d(x)v^q\right) + \nu(x)\big] \qquad &\mbox{ in }&\Omega,
 \endalignedat}\right.
 \end{equation}
 where $\Omega\subseteq\rn$,
\begin{equation}
 \label{cond_pq}
p\ge0,\qquad q>0,\qquad q\ge |1-p|,
 \end{equation}
 and the coefficients $a,b,c,d,\mu,\nu$ are H\"older continuous functions in
 $\overline{\Omega}$, with
 \begin{equation}
\label{cond_abcd}
a,b>0,\qquad c,d\ge0,\qquad ab-cd\ge 0\;\mbox{ in }\;\overline{\Omega}.
\end{equation}

Cases of particular interest are: \begin{itemize}
\item[(i)] $p=0$, $q=1$ -- then \re{model} is a Lotka-Volterra system, a model on biological species interactions;
\item[(ii)] $p=0$, $q=2$ -- then \re{model} models phenomena in nonlinear optics and the theory of Bose-Einstein condensates;
\item[(iii)] $p=1$, $q=1$ -- in this case \re{model} is related to a  model of chemical reaction. More general models in chemistry are obtained by varying $p$ and $q$.
    \end{itemize}
 We refer to Section 1.3 in \cite{MSS} for a more detailed discussion on these applications, as well as references. We observe the last condition in \re{cond_abcd} means the reaction in the system dominates the absorption, so there is no conservation of mass in the time-dependent version of \re{model}.
In \cite{MSS} we studied the classification and non-existence of positive solutions of \re{model} in unbounded domains, as well as the a priori bounds and existence results which can be inferred in smooth bounded domains.

 A major difficulty in the study of system \re{model} is that it is in general neither {\it cooperative} (or quasi-monotone), nor {\it variational} in the sense that its solutions cannot be written as critical points of some functional defined on a Banach space. The core of the new method developed in \cite{MSS} consists in proving Liouville type theorems for a system of elliptic inequalities satisfied by some auxiliary sub- and super-harmonic functions of $u,v$. This method gives many new results even for systems which happen to be variational, such as (i) and (ii) above with $a=d$.

 However, the proofs of the results in \cite{MSS} depend crucially on the fact that the second-order elliptic operator in \re{model} is the Laplacian, that is, in the models above only homogeneous diffusion can be considered. Our main goal here is to remove this hypothesis and show that the main results  in \cite{MSS} are valid for general operators in non-divergence form. It is remarkable that the (necessarily) different proofs we give here not only permit to generalize but also to shorten some of the proofs from \cite{MSS}. Our arguments will be entirely based on the maximum principle and its consequences from the regularity theory for elliptic equations.

 In terms of the applications, if the underlying stochastic process $X_t$ is not a pure diffusion (i.e. a Brownian motion $W_t$) but rather follows  $dX_t = \Sigma(X_t) dW_t$, for some positive variance matrix $\Sigma$, in the corresponding PDE the Laplacian is replaced by $tr(A(x) D^2u)= \sum a_{ij}(x)\partial_{ij}u$, where the matrix $A=\Sigma^T\Sigma$ accounts for the spatial inhomogeneity. If the process is also allowed to have drift $dX_t = b(X_t)dt+\Sigma(X_t) dW_t$, then we end up with the differential operator with a first order term
 \begin{equation}\label{linear}
 Lu=tr(A(x) D^2u) + b(x).Du= \sum a_{ij}(x)\partial_{ij}u + \sum b_i(x) \partial_i u.
  \end{equation}
  Replacing the Laplacian by such operators in the examples (i)-(iii) above means allowing heterogeneous media, as well as a possibility to consider gradient-dependent equations, i.e. advection in addition to  diffusion and reaction-absorption (for (i) and (iii)), or more general derivative Schr\"odinger equations (for (ii), see for instance \cite{B}, chapter I.6).

Even more generally, an object of intensive study are controlled processes, in which $X_t$ follows
$dX_t = b_{\alpha_t}(X_t) dt + \sigma_{\alpha_t}(X_t)dW_t$, where $\alpha_t$ is an index process corresponding to a choice made in order to maximize or minimize some cost function (see \cite{bardi}, \cite{FS}). The PDE operators modeling such processes are suprema or infima of linear operators as in \re{linear}, with fixed ellipticity constants and bounds for the coefficients. These operators are usually referred to as Hamilton-Jacobi-Bellman (HJB) operators, and are in turn a subclass of the so-called Isaacs min-max operators, basic in game theory.
In the following we will consider Isaacs operators as in \re{isaacs} below, that is, sup-inf over arbitrary index sets of linear operators as in \re{linear}
\begin{equation}\label{isaacs}
 \F[u]:=\displaystyle\mathop{\sup}_{\alpha\in {\cal
A}}\mathop{\inf}_{\beta\in {\cal B}}L^{(\alpha,\beta)}u =\displaystyle\mathop{\sup}_{\alpha\in {\cal
A}}\mathop{\inf}_{\beta\in {\cal B}} \left(\sum_{i,j} a^{(\alpha,\beta)}_{ij}(x)\partial_{ij}u + \sum_i
b_{i}^{(\alpha,\beta)}(x)
\partial_i u\right)\,;
\end{equation}
here ${\cal A}$, ${\cal B}$ are arbitrary index sets and the coefficients $a^{(\alpha,\beta)}_{ij}$, $b_{i}^{(\alpha,\beta)}$ are continuous functions. We assume that all eigenvalues of the symmetric matrices $A^{ (\alpha,\beta)}$ belong to a fixed interval $[\lambda,\Lambda]$, and that the $L^\infty$-norms of the vectors $b^{(\alpha,\beta)}$ are bounded by $B$, for some constants  $0<\lambda\le \Lambda$, $B\ge0$. Note this is equivalent to Definition~\ref{def32} below.
Note also that
$\F$ is linear and reduces to \re{linear} when $|{\cal A}|=|{\cal B}|=1$ in \re{isaacs}, and that
$\F$ is a HJB operator when $|{\cal A}|=1$ or $|{\cal B}|=1$. In the sequel we will write $\F[u]= F(D^2u, Du,x)$ when we want to distinguish the dependence of $\F$ in the derivatives of $u$ and $x$. Writing $F(D^2u)$ will mean $F$ is autonomous and depends only on the second derivatives of $u$.

Throughout the paper solutions are understood in the viscosity sense. By applying the well-known regularity results for viscosity solutions to the systems we consider, we know that their solutions are in $C^{1,\gamma}$ for some $\gamma>0$, and even in $C^{2,\gamma}$ provided the operator $\F$ is of HJB type and the coefficients in the equation are H\"older continuous. We observe that viscosity solutions are not an added complication, they provide a good framework in this setting, just like Sobolev spaces do for some divergence-form operators, even when one knows that any $H^1$-solution is classical.
\bigskip

 Next, we comment on the novelty of our results. First and foremost, they are the first of their kind for models like (i)-(iii) above with {\it controlled diffusions}, that is, for systems like \re{model} where the Laplacian is replaced by a fully nonlinear operator such as a Isaacs (or even HJB) operator.

Let us give some more context for each of the models (i), (ii), (iii). Linear operators as in \re{linear} were considered for Lotka-Volterra systems in a number of papers; the most general results available to date as well as references can be found in \cite{DLS}. When reduced to the linear case, the theorems below  strengthen the results from \cite{DLS}.

The comparison between the case of the Laplacian and more general operators is probably most easily made for (ii). For the last ten years there has been a huge amount of work for systems in the form (ii), and they all assume the differential operator is in divergence form, in particular the Laplacian. When reduced to (ii), our main results from \cite{MSS} completed the previous works on Schr\"odinger systems, by establishing existence results for the case $a,b,c,d\ge0$, which was almost unstudied. On the other hand, Theorem \ref{boun} below seems to be the first result whatsoever for Schr\"odinger systems governed by inhomogeneous diffusions, i.e. with non-divergence form linear elliptic operators, independently of the sign of the coefficients $a,b,c,d$.

Finally, for (iii), the results in \cite{MSS} seem to be the first on their kind. Here we extend these to spatially inhomogeneous and controlled diffusions, which have obvious relevance in chemistry.
\bigskip

 We observe that existence and non-existence for fully nonlinear  systems with a different class of  nonlinearities (with cooperative and fully coupled leading terms such as Lane-Emden type systems, see \re{lanesystemcone} below) can be found in \cite{QSsys}.  We refer to that paper, as well to \cite{BS} and \cite{IK}, for more examples and references on problems where fully nonlinear systems appear. In passing, and in order to provide a quotable source, below we will record several general nonexistence results for systems of Lane-Emden type, which essentially follow from a number of recent advances in the theory of Isaacs operators but do not seem to have appeared before. We will use some of these results in our proofs.
\bigskip

The following classification theorem is our first main result. It represents a strong rigidity property, and states that  nonnegative entire solutions $(u,v)$ of the  system
\begin{equation}
\label{modelf}
\left\{\ {\alignedat2
 -F(D^2u)  &=u\Big[v^p\left(av^q-cu^q\right)\big] \\
  -F(D^2v) &=v\Big[u^p\left(bu^q-dv^q\right)\big]
 \endalignedat}\right.
 \end{equation}
 can only be semi-trivial (i.e. $u\equiv0$ or $v\equiv0$) or have proportional components (i.e. $u/v=$const). In other words, rather unexpectedly, the existence of a positive entire solution of \re{modelf} is equivalent to the existence of a positive entire  solution of a scalar equation.

 \begin{thm}\label{nonexrn} Assume $a,b,c,d,p,q$ are real numbers such that \re{cond_pq} and \re{cond_abcd} hold.

 If $u,v$ are nonnegative functions which satisfy \re{modelf} in the whole space $\rn$, then either $u\equiv0$, or $v\equiv0$, or there exists a real number $K>0$ such that $u\equiv Kv$.
 \end{thm}

A particularly remarkable feature of this result is that it is independent of any notion of criticality, that is, of how large $p$, $q$ or $n$ are.

Theorem \ref{nonexrn}   reduces the question of existence of positive entire solutions of the system \re{modelf} to that of the scalar equations
\begin{equation}\label{scal}
-F(D^2u)  =0\qquad\mbox{and}\qquad -F(D^2u)  =u^{p+q+1}\qquad\mbox{in }\; \rn.
\end{equation}
 We discuss the available nonexistence results for these equations below.
\bigskip

Next, we consider the question of existence of positive solutions of \re{modelf} in cones of~$\rn$. By a cone we mean a set in the form $\co=\{x\in\rn\setminus\{0\}\::\: x/|x|\in \omega\}$, for some $C^2$-smooth subdomain of the unit sphere $\omega$.

\begin{thm}\label{nonexhs}
Assume $a,b,c,d,p,q$ are real numbers such that \re{cond_pq} and \re{cond_abcd} hold.
Let $u,v$ be nonnegative functions which satisfy \re{modelf} in a nonempty open cone $\co\subset \rn$, and are proportional or vanish on $\partial\co$.

(a) If $u$ and $v$ are bounded, then $u$ and $v$ are proportional, or $u$ and $v$ vanish.

(b) If $c=d=0$ then $u$ and $v$ are proportional, or one of them vanishes.
\end{thm}

This result is independent of how large $p$ and $q$ are, and reduces the system to the equations \re{scal}, set in $\co$.

Note that results on classification of bounded solutions of fully nonlinear systems in cones were previously obtained in \cite{QSsys}, for a different type of nonlinearities (of Lane-Emden type), and with different proofs. On the other hand, the statement (b) above is to our knowledge the first classification result for unbounded solutions of systems with operators in nondivergence form. We also observe this result covers nonlinearities such as $uv^2$, $u^2v$, which were recently found to play an important role in some applications, see for instance \cite{BW}. \smallskip

Finally, we state the existence results in bounded domains which we can obtain as consequences of the previous theorems and well-known techniques from Leray-Schauder-Krasnoselskii degree theory.

We introduce the following notation. For any orthogonal matrix $Q$ and any fixed $y\in\overline{\Omega}$ we denote with $\ofq$ the pure second order operator defined by
$$
\ofq(D^2u):= F(Q^T\,D^2u\,Q,0,y).
$$
If $F$ is a linear operator as in \re{linear}, then it is easy to see that for each $y$ there exists $Q$ such that $\ofq(D^2u)=\Delta u$. If $F$ is rotationally invariant (such as, for instance, an extremal Pucci operator -- see Definition \ref{def31} below), then $\ofq=F$. Recall that  $F(D^2u)$ is  rotationally invariant  if it only depends on the eigenvalues of $D^2u$.

In the following we assume that $\Omega$ is a bounded Lipschitz domain such that each point $y$ on the boundary $\partial\Omega$ has a neighbourhood in $\overline{\Omega}$ which is $C^2$-diffeomorphic to a neighbourhood of the origin in some closed cone $\overline{\C_{\omega_y}}$. Observe if $\partial \Omega$ is $C^2$-smooth then every such cone is a half-space.

\begin{thm}\label{boun} Let $\Omega$ be as stated, and
assume the coefficients $a,b,c,d,\mu,\nu,p,q$ are  such that \re{cond_pq} and \re{cond_abcd} hold and
\begin{equation}
\label{hypodi}\inf_\Omega (ab-cd)>0,\qquad \mu,\nu\le0\quad\mbox{in }\;\Omega.
\end{equation}

Assume in addition that for each $y\in \overline{\Omega}$ there exists an orthogonal matrix $Q$ such that if $u$ is a   bounded nonnegative solution of the equation
\begin{equation}
\label{hypoli}
-\ofq(D^2u)=u^{p+q+1} \quad\mbox{in } \rn\mbox{ or } \; \C_{\omega_y},\quad\mbox{then}\quad u\equiv0,
\end{equation}
(see Proposition \ref{propoli} below).
Then the system \re{model}, with the Laplacian replaced by any general Isaacs operator as in \re{isaacs}, has a positive solution in $\Omega$ with $u=v=0$ on $\partial\Omega$.
\end{thm}

It is worth observing that we can consider Lipschitz domains almost ``free of charge", instead of only smooth domains. This is also a new feature with respect to previous works on these types of systems, even for systems with Laplacians.


The first hypothesis in \re{hypodi} cannot be weakened, even for simple systems with the Laplacian (see the remark after Corollary 1.4 in \cite{MSS}). As will be clear from the proof, instead of assuming that $\mu,\nu\le0$ in \re{hypodi} we could suppose that $\mu, \nu$ are smaller than the (positive) first semi-eigenvalue of the Pucci maximal operator, or than the first semi-eigenvalue of $F$ itself, if $F$ is a HJB operator (Definition \ref{def33}).

A discussion of hypothesis \re{hypoli} is in order.

In general, the only non-existence results in unbounded domains for the equation in \re{hypoli} concern positive supersolutions. Optimal results for nonexistence of positive supersolutions of \re{hypoli} in the whole space or in cones are known, \cite{CL}, \cite{armsirfirst}, \cite{armsir}, \cite{L}, \cite{ASS}, and can be expressed in terms of the so-called scaling exponents of the operator $\ofq$ in \re{hypoli}. These results will be discussed in detail in the next section, we will record here their consequences which apply to hypothesis~\re{hypoli}.

 Given an Isaacs operator $F(D^2u)$, we denote with $\alpha^*(F)$ the ''scaling exponent" of $F$, defined as  the supremum of all positive $\alpha$ such that $-F(D^2u)\ge0$ has a $(-\alpha)$-homogeneous solution in $\rn\setminus\{0\}$ (if such $\alpha$ exist; if not, then even $-F(D^2u)\ge0$ has no positive supersolutions in $\rn$). Similarly,  we denote with  $\alpha^+(F,\co)$  the supremum of all positive $\alpha$ such that $-F(D^2u)\ge0$ has a $(-\alpha)$-homogeneous solution in $\co\setminus\{0\}$ (such $\alpha$ always exist) See \cite{ASS2} and \cite{ASS} for more details on these scaling exponents and the way they describe properties of the operators $F$. Obviously $\alpha^*(\ofq)\le \alpha^+(\ofq,\C)$.

It is known that $\alpha^*(\Delta)= n-2$, $\alpha^+(\Delta,\rn_+)=n-1$, and for every Isaacs operator
\begin{equation}
\label{bounds1}
\frac{\lambda}{\Lambda}(n-1)-1\le\alpha^*(F) \le \frac{\Lambda}{\lambda}(n-1)-1,\qquad \frac{\lambda}{\Lambda}n-1\le\alpha^+(F,\rn_+)\le \frac{\Lambda}{\lambda}n-1.
\end{equation}
These bounds are obtained by evaluating the  scaling exponents for the Pucci extremal operators, and first appeared in \cite{CL} and \cite{L}.

The following proposition contains necessary conditions for \re{hypoli}, which are direct consequences from the Liouville results for the equation in \re{hypoli}  in the next section.
\begin{propo}\label{propoli}
The hypothesis \re{hypoli} in Theorem \re{boun} is satisfied if\begin{itemize}
 \item[(i)] $\displaystyle p+q\le \frac{2}{\alpha^+(\ofq,\C_{\omega_y})}$, $y\in \partial\Omega$ and $\displaystyle p+q\le \frac{2}{\max\{\alpha^*(\ofq), 0\}}$, $y\in \Omega$, or

\item[(ii)] $\Omega$ is a $C^2$-domain, $\overline{F}_y(D^2u)=F(D^2u,0,y)$ is a rotationally invariant operator for each $y\in \overline{\Omega}$, and
$\displaystyle p+q\le \frac{2}{\max\{\alpha^*(\overline{F}_y), 0\}}$, or

\item[(iii)] $\Omega$ is a $C^2$-domain, $F$ is linear as in \re{linear}, and $\displaystyle  p+q<\frac{4}{n-2}$.
\end{itemize}
\end{propo}

Observe the upper bounds in this proposition are void, if the scaling exponent $\alpha^*$ is nonpositive (this is the case for instance for the Pucci maximal operator if $\frac{\Lambda}{\lambda}\ge n-1$). For any given operator the scaling exponents can be evaluated in terms of (bounds on) the coefficients of the operator, by constructing homogeneous sub- and super-solutions, and the bounds thus obtained can be combined with Proposition \ref{propoli}. The most general bounds are given in \re{bounds1}, and combining \re{bounds1} with Proposition \re{propoli} gives explicit relations between $\lambda$, $\Lambda$, $n$, $p$, $q$, such that \re{hypoli} is satisfied. In the linear case, Proposition \ref{propoli} (iii), the result is particularly simple to state.\smallskip

The paper is organized as follows. In the next section we state more general results which follow from our proofs, and list various non-existence results for elliptic inequalities and systems in unbounded domains. Section \ref{sect3} contains some easy or well-known preliminary results; while in Section \ref{sect4} we prove nonexistence results in unbounded domains for a class of systems that will be derived from the systems we consider. Section \ref{sect5} contains the proofs of the classification theorems in unbounded domains (Theorems \ref{nonexrn} and \ref{nonexhs}, and their more general versions, Theorems \ref{thm_proportionnalite_hn}, \ref{newhs} and \ref{thm_proportionnalite_rn}) while in Section \ref{sect6} we give the proof of the existence result, Theorem \ref{boun} (and its more general version, Theorem \ref{AprioriBound2}).

 \section{More general results and Liouville theorems}\label{sect2}

 Our proofs yield more general results than the ones stated in the introduction. We list these more general theorems in this section. We also discuss here Liouville type results for scalar inequalities and systems in unbounded domains.

The system \re{modelf} is included in the class of systems
\begin{equation}
\label{mainsyst}
\left\{\ {\alignedat2
 -\F[u]&=f(x,u,v),&\qquad x\in \Omega\\
 -\F[v]&=g(x,u,v),&\qquad x\in \Omega,
 \endalignedat}\right.
 \end{equation}
where the main feature of the nonlinearities $f$ and $g$ (or their leading order terms) is that they satisfy the  condition
\begin{equation}
\label{condition_fg}
\exists\: K>0\::\:\quad [f(x,u,v)-Kg(x,u,v)][u-Kv]\leq 0  \quad \hbox{ for all $(u,v)\in \mathbb{R}^2, \;x\in \Omega$. }
\end{equation}
Indeed, we recall the following result from \cite{MSS} (Proposition 1.3 in that paper).

\begin{propo}[\cite{MSS}]\label{propalg} If $f= u^rv^p[av^q-cu^q]$ and $g= v^ru^p[bu^q-dv^q]$, where the real parameters $a,b,c,d,p,q,r$ satisfy
\begin{equation}\label{Hyp_abcdpqr}
a, b>0,\quad c,d\ge 0,\quad ab\ge cd, \qquad p,r\ge 0,\quad q>0,\quad q\ge |p-r|.
\end{equation}
then $f$ and $g$ satisfy \re{condition_fg}, for a unique number $K$ such that $a-cK^q\ge 0$, $bK^q-d\ge0$.
\end{propo}

\subsection{Liouville theorems in cones}

     We assume here that $\Omega$ is the cone
     $ \co=\{tx, t>0, x\in \omega\}$,
     where $\omega$ is a $C^2$-smooth subdomain of the unit sphere in $\rn$. The following theorem contains Theorem \ref{nonexhs}.

\begin{thm}
\label{thm_proportionnalite_hn}
Let $\F$ be an Isaacs operator as in \re{isaacs}
and  $f$ and $g$ satisfy (\ref{condition_fg}).
Let $\Omega=\co$ and
 $(u,v)$ be a bounded solution of \re{mainsyst} such that $u\equiv Kv$ on $\partial\co$.

 Then $u\equiv Kv$ in $\co$.
\end{thm}

This theorem will be obtained as a consequence of the Phragm\`en-Lindel\"of principle for fully nonlinear equations.
 Theorem \ref{thm_proportionnalite_hn} reduces the question of existence of positive solutions of \re{mainsyst} in a cone to the very important in itself question of solving the scalar equation
\begin{equation}\label{scalarcone}
-\F[u]= f(x, u)\ge0, \quad u>0,\quad\mbox{in }\;\co.
\end{equation}

Two types of nonexistence results are available for this equation: general results for supersolutions and more precise results for solutions when the domain is a half-space and the operator $\F$ is rotationally invariant. As far as the latter case is concerned, it is proved in \cite{D} and Theorem 3.1 in \cite{QSeq} that the nonexistence of solutions of the equation $-F(D^2u)=f(u)$  in $\rn$ (and even in $\mathbb{R}^{n-1}$) implies the nonexistence of positive solutions in a half-space of $\rn$, for every rotationally invariant operator $F$ and locally Lipschitz nonlinearity $f$ (in \cite{QSeq} only Pucci operators were considered, but the proof is the same for every rotationally invariant operator). An extension of this result to systems is proved in \cite{QSsys}. Even stronger results are known if the elliptic operator is the Laplacian, \cite{CLZ}.

 Next, a nearly optimal result for supersolutions can be obtained by combining the results and methods from the recent papers \cite{armsir} and \cite{ASS}. In order to provide a quotable source, we state a rather general version of this nonexistence theorem.

\begin{thm}\label{lioucone}
Let $F(D^2u)$ be an Isaacs operator, and $b\ge0$. Set
$$
\F[u]:= F(D^2u) + \frac{b}{|x|}|Du|.
$$
Let $\alpha^+= \alpha^+(\F,\co)>0$ and $\alpha^-= \alpha^-(\F,\co)<0$ be respectively the supremum and the infimum of all $\alpha\in \R$ such that $-\F[u]\ge0$ has a positive $(-\alpha)$-homogeneous supersolution in $\co\setminus\{0\}$ (as in Section 3 of \cite{ASS}). Let $\gamma<2$ and
\begin{equation} \label{sigmadef}
\sigma^- : = 1 + \frac{2-\gamma}{\alpha^-}<1\,,\qquad \sigma^+ : = 1 + \frac{2-\gamma}{\alpha^+}>1.
\end{equation}
Assume that the function $g:(0,\infty) \to (0,\infty)$ is continuous and satisfies
\begin{equation} \label{liouvcond}
\liminf_{s\searrow 0}\; s^{-\sigma^+} g(s) > 0 \qquad \mbox{and} \qquad \liminf_{s\to \infty}\; s^{-\sigma^-} g(s) > 0,
\end{equation}
while the continuous function $h:\co\setminus B_{R_0}\to (0,\infty)$ is such that  for some $\omega^\prime\subset \omega$ and $c_0>0$
$$
h(x)\ge c_0|x|^{-\gamma}\qquad \mbox{in }\; \mathcal{C}_{\omega^\prime}\setminus B_{R_0}.
$$
Then for every $R_0>0$ the differential inequality
\begin{equation}\label{antico}
-\F[u]\geq h(x) g(u)\qquad \mbox{in }\; \co\setminus B_{R_0}
\end{equation}
does not have a positive solution.
\end{thm}
\begin{proof} This theorem is proved by repeating the proof of Theorem 5.1 in \cite{armsir}, replacing the functions $\Psi^+$ and $\Psi^-$ there by the corresponding functions constructed in \cite{ASS} and by making use of the comparison principle for $\F$.
\end{proof}

We list several consequences.
\begin{thm}\label{liouconegrad}
Let $F(D^2u)$ be an Isaacs operator and $B\ge0$. Let $\alpha^+= \alpha^+(F,\co)$, $\alpha^-= \alpha^-(F,\co)$, and $\sigma^+$, $\sigma^-$ be defined as in \re{sigmadef}. The differential inequality
\begin{equation}\label{antico1}
-F(D^2u) + B|Du|\geq h(x) g(u)\qquad \mbox{in }\; \co\setminus B_{R_0}
\end{equation}
does not have a positive solution, provided $g$ and $h$ are as in the previous theorem, with \re{liouvcond} replaced by the slightly stronger hypothesis: for some $\varepsilon>0$
\begin{equation} \label{liouvcondgrad}
\liminf_{s\searrow 0}\; s^{-\sigma^++\varepsilon} g(s) > 0 \qquad \mbox{and} \qquad \liminf_{s\to \infty}\; s^{-\sigma^--\varepsilon} g(s) > 0,
\end{equation}
\end{thm}
\begin{proof}
This is a consequence of the previous theorem, if we observe that for every $b>0$ there is $R_0>0$ such that $B\le b/|x|$ for $|x|>R_0$ and take $b$ so small that the scaling exponents of $F(D^2\cdot) +(b/|x|)|D\cdot|$ are sufficiently close to $\alpha^\pm(F,\co)$, and \re{liouvcond} is satisfied.
\end{proof}

Next, we give two Liouville theorems for systems in cones.

The following result, which applies to the systems from the introduction, is an obvious consequence of  Theorem \ref{thm_proportionnalite_hn} and Theorem \ref{liouconegrad}.
\begin{cor}
Under the hypothesis of Theorem \ref{thm_proportionnalite_hn}, if $\F[u]\ge F(D^2u) -B|Du|$ and there exist
$c_0,\sigma>0$ such that for all $x\in \hn$ and $v\geq 0$,
$$f(x,Kv,v)=c_0\,v^\sigma,\qquad\mbox{with }\; \sigma <1+\frac{2}{\alpha^+(F,\co)},$$
then the only nonnegative bounded solution of (\ref{mainsyst}) in $\co$ is the trivial one.
\end{cor}

Finally, we record the following Liouville theorem for the so-called {\it fully nonlinear Lane-Emden system}. More general right-hand sides can be readily studied by the same argument (given in Section 6 of \cite{armsir}).
\begin{thm}\label{laneemdencone}
Let $F_1(D^2u)$, $F_2(D^2u)$ be Isaacs operators with scaling exponents $\alpha_1^+$, $\alpha_2^+$ in a cone $\co$. Let $r,s\ge0$. The only nonnegative solution of the system
\begin{equation}\label{lanesystemcone}
\left\{\ {\alignedat2
 -F_1(D^2u) +B|Du|  &\ge v^r \\
  -F_2(D^2u) +B|Du|  &\ge u^s
 \endalignedat}\right.\qquad\mbox{in }\;\co\setminus B_{R_0}
 \end{equation}
is the trivial one, provided
$$
rs\le 1\qquad\mbox{or}\qquad\frac{2(1+r)}{rs-1}< \alpha_1^+\qquad\mbox{or}\qquad \frac{2(1+s)}{rs-1}<\alpha_2^+.
$$
If $B=0$ weak inequalities can be allowed in the last hypothesis.
\end{thm}
\begin{proof} Repeat the proof on page 2041 in \cite{armsir} where the whole space instead of a cone was considered. Replace the references to Lemma 3.8 there by references to Lemma 5.4 in that paper, and as above, replace the functions $\Psi^+$, $\Psi^-$ by the more general functions of this type, constructed in \cite{ASS} for arbitrary Isaacs operator. If $B\not=0$, reason as in the proof of  Theorem \ref{liouconegrad} above.
\end{proof}

Finally, we state a theorem for nonexistence of unbounded solutions in cones for the type of systems we are mostly interested in this paper, of which Theorem \ref{nonexhs} (b) is a very particular case.

\begin{thm}
\label{newhs}
Let $p,q,r,s\geq 0$.
We assume that $f,g$ satisfy condition (\ref{condition_fg}) for some constant $K>0$ and that, for some $c>0$,
\begin{equation}
\label{condition_fg2}
f(x,u,v)\geq c \;u^r \;v^{p} \text{\quad  and \quad} g(x,u,v)\geq c\; u^{q}\; v^s \quad \text{for all $u,v\geq 0$ and $x\in \co$}.
\end{equation}
Let $(u,v)$ be a nonnegative classical solution of (\ref{mainsyst}) in $\mathbb{R}^n_+$,
such that $u=Kv$ on $\partial \co$. Let $\F[u]=F(D^2u)$ for some Isaacs operator $F$.
\begin{itemize}
\item [(i)] Either $u\leq Kv$ or $u\geq Kv$ in $\mathbb{R}^n_+$.
\item [(ii)] If
\begin{equation}
\label{condition_rnsd}
 r \leq 1+ \frac{2}{\alpha^+}- p\frac{\alpha^-}{\alpha^+}\quad\hbox{or}\quad s \geq 1+ \frac{2}{\alpha^-}- q\frac{\alpha^+}{\alpha^-},
\end{equation}
and
\begin{equation}
\label{condition_rnsd2}
 s\leq 1+ \frac{2}{\alpha^+}- q\frac{\alpha^-}{\alpha^+}\quad\hbox{or}\quad r \geq 1+ \frac{2}{\alpha^-}- p\frac{\alpha^+}{\alpha^-},
\end{equation}
then either
$u\equiv Kv $ or $ (u,v) $  is semi-trivial. Here $\alpha^+= \alpha^+(F,\co)>0$ and $\alpha^-= \alpha^-(F,\co)<0$.
\end{itemize}
\end{thm}

   \subsection{Classification results in the whole space, $\Omega=\rn$}

For   $\Omega=\rn$, we focus on the following system
\begin{equation}
\label{system_Spqr}
\left\{\quad{\alignedat2
-F(D^2 u)&=u^rv^p[av^q-cu^q] &\text{\qquad on }\mathbb{R}^n\\
 -F(D^2 v)&=v^ru^p[bu^q-dv^q]&\text{\qquad on }\mathbb{R}^n,
 \endalignedat}\right.
\end{equation}
where we always assume that the real parameters $a,b,c,d, p,q,r$ satisfy the hypothesis \re{Hyp_abcdpqr} of Proposition \ref{propalg} (this hypothesis reduces to \re{cond_pq} and \re{cond_abcd} when $r=1$).

Recalling that $\alpha^*(F)$ was defined in the previous section, we have the following result, in which Theorem \ref{nonexrn} is contained.
   \begin{thm}
\label{thm_proportionnalite_rn}
Let $F(D^2u)$ be an Isaacs operator,  \re{Hyp_abcdpqr} holds, and $K$ be the number given by Proposition \ref{propalg}.
Let $(u,v)$ be a positive viscosity solution of (\ref{system_Spqr}) in $\rn$.
\begin{itemize}
\item[i)] Assume that
$$ \alpha^*(F)\leq 0 \quad\text{ or }\quad  0\leq r\leq 1+ \frac{2}{\alpha^*(F)}.$$
If $p+q<1$, we assume moreover that $u$ and $v$ are bounded. Then
$ u\equiv Kv$.
\item[ii)] Assume that
$$ \alpha^*(F)\leq 0 \quad\text{ or }\quad \left(\;p\leq \frac{2}{\alpha ^*(F)} \quad\text{ and }\quad c,d>0\;\right).$$
If $q+r\leq 1$, we assume moreover that $u$ and $v$ are bounded. Then $ u\equiv Kv$.
\end{itemize}
\end{thm}

Observe in this theorem there is no restriction on the total degree $\sigma=p+q+r>0$ of the system \re{system_Spqr}. The exponents $p$ and $r$ separately can guarantee that (\ref{system_Spqr}) has no nonstandard solutions.

An easy consequence is the following Liouville type result for the system~\re{system_Spqr}.
   \begin{thm}
\label{thm_liouville_sol_positive_rn}
Under the hypotheses of the previous theorem, if $ab>cd$ and
\begin{equation}
\label{hypothese_Liouville}
\text{$-F(D^2u)=u^{p+q+1}$ has no bounded positive viscosity solution on $\rn$,}
\end{equation}
 then (\ref{system_Spqr}) has no bounded positive viscosity solution in $\rn$.
\end{thm}

As far as the  hypothesis \re{hypothese_Liouville} is concerned, we recall the following result from \cite{armsir}.

\begin{thm}[\cite{armsir}] \label{lioueq}
Let $F(D^2u)$ be an Isaacs operator. The equation
$$-F(D^2u)=u^\sigma$$
has no positive supersolutions in $\rn$ (and even in any exterior domain in $\rn$) provided
$$ \alpha^*(F)\leq 0 \quad\text{ or }\quad  0\leq \sigma\leq 1+ \frac{2}{\alpha^*(F)}.$$
\end{thm}

It is worth observig that it is an outstanding open question whether the ranges of $\sigma$ for which the equation in \re{hypoli} (for instance, if $F$ is a Pucci operator) does not admit entire positive supersolutions or entire positive solutions are different. Note this fact is well-known for the Laplacian -- the equation $-\Delta u = u^\sigma$ does not have positive entire solutions if and only if $\sigma < (n+2)/(n-2)$, while this equation does not have positive entire supersolutions if and only if $\sigma\le n/(n-2)$. For Pucci operators the only result in that direction is \cite{FQ5}, and it concerns only radial (super)solutions.

We also record the extension of the last theorem to systems of Lane-Emden type, proved in  Section 6 of \cite{armsir}.

\begin{thm}\label{laneemdencone1}
Let $F_1(D^2u)$, $F_2(D^2u)$ be Isaacs operators with scaling exponents $\alpha_1^*$, $\alpha_2^*$ in $\rn$. Let $r,s\ge0$. The only nonnegative solution of the system
\begin{equation}\label{lanesystemcone1}
\left\{\ {\alignedat2
 -F_1(D^2u)   &\ge v^r \\
  -F_2(D^2u)   &\ge u^s
 \endalignedat}\right.\qquad\mbox{in }\;\rn\setminus B_{R_0}
 \end{equation}
is the trivial one, provided $\alpha_1^*\le0$, or $\alpha_2^*\le0$, or
$$
rs\le 1\qquad\mbox{or}\qquad\frac{2(1+r)}{rs-1}< \alpha_1^*\qquad\mbox{or}\qquad \frac{2(1+s)}{rs-1}<\alpha_2^*.
$$
\end{thm}

In the last theorem in this section we discuss the classification of nontrivial nonnegative solutions of \re{system_Spqr}.
\begin{thm}
\label{thm_liouville_rn}
Under the hypothesis of Theorem \ref{thm_liouville_sol_positive_rn}, if we moreover assume that
$q+r>1$,
then any nonnegative bounded viscosity solution of (\ref{system_Spqr}) is semitrivial, i.e.
$$(u,v)=(C_1,0) \text{\qquad or \qquad } (u,v)=(0,C_2)$$ with $C_1,C_2\geq 0$.
Moreover:\\
- If $r=0$, then $(u,v)=(0,0)$. \\
- If $r>0$, $p=0$ and $c>0$ (resp. $d>0$), then $(u,v)=(0,0)$.
\end{thm}

\subsection{A priori estimates and existence in a bounded domain}

   We consider the following system with  general lower order terms
\begin{equation}
\label{intro_system_pqr_Dir2}
\left\{\quad{\alignedat2
-\F[u]&=u^rv^p\bigl[a(x)v^q-c(x)u^q\bigr]+h_1(x,u,v), &\qquad&x\in \Omega,\\
-\F[v]&=v^ru^p\bigl[b(x)u^q-d(x)v^q\bigr]+h_2(x,u,v), &\qquad&x\in \Omega,\\
u&=v=0, &\qquad&x\in \partial\Omega,\\
 \endalignedat}\right.
\end{equation}
where $\F$ is a general Isaacs operator as in \re{isaacs} and $\Omega$ is a bounded Lipschitz domain such that each point $y$ on the boundary $\partial\Omega$ has a neighbourhood in $\overline{\Omega}$ which is $C^2$-diffeomorphic to a neighbourhood of the origin in some closed cone $\overline{\C_{\omega_y}}$.

The following theorem contains Theorem \ref{boun} as a very particular case.
\begin{thm}
\label{AprioriBound2}
Let $\F$ be an Isaacs operator, and $p, r\ge 0$, $q>0$, $q\ge |p-r|$, $q+r>1$.
Assume that the system (\ref{system_Spqr}) has no bounded positive viscosity solution in $\rn$ or in any cone as in the definition of the Lipschitz property of $\Omega$ above (sufficient conditions for this are given in the previous subsection).

1. Let $a, b, c, d\in C(\overline\Omega)$ satisfy $a, b>0$, $c, d\ge 0$ in $\overline\Omega$ and
\begin{equation}
\label{hypApriori2}
\inf_{x\in \Omega}\,\left[a(x)b(x)-c(x)d(x)\right]\,>0.
\end{equation}
Let $h_1, h_2\in C(\overline\Omega\times [0,\infty)^2)$ satisfy
\begin{equation}
\label{hypApriori3}
\lim_{u+v\to\infty} \frac{h_i(x,u,v)}{(u+v)^\sigma}=0,\quad i=1,2,
\end{equation}
and let one of the following two sets of assumptions be satisfied:
\begin{equation}
\label{hypApriori4}
\left\{\quad{\alignedat2
&r\le 1,\quad \hbox{and, setting }\: \bar m := \min\{\inf_{x\in\Omega} a(x), \inf_{x\in\Omega} b(x)\}>0,\\
&\liminf_{v\to\infty,\ u/v \to 0} \frac{h_1(x,u,v)}{u^rv^{p+q}} >-\bar m,\quad
\liminf_{u\to\infty, \ v/u \to 0} \frac{h_2(x,u,v)}{v^ru^{p+q}} >-\bar m,
\endalignedat}\right.
\end{equation}
or
\begin{equation}
\label{hypApriori5}
\left\{\quad{\alignedat2
&m:=\min\{\inf_{x\in\Omega} c(x), \inf_{x\in\Omega}d(x)\}>0, \quad\mbox{ and } \\
&\limsup_{u\to\infty,\ v/u \to 0} \frac{h_1(x,u,v)}{u^{r+q}v^p} <m,\quad
\limsup_{v\to\infty,\ u/v \to 0} \frac{h_2(x,u,v)}{v^{r+q}u^p} <m
 \endalignedat}\right.
\end{equation}
(with uniform limits with respect to  $x\in\overline\Omega$ in (\ref{hypApriori3})--
(\ref{hypApriori5})). \\
Then there exists $M>0$ such that any positive classical solution $(u,v)$ of
(\ref{system_pqr_Dir2}) satisfies
\begin{equation}
\label{AprioriBoundM}
\sup_\Omega u \leq M,\qquad \sup_\Omega v \leq M.
\end{equation}

2. Assume  that \re{hypApriori2}--\re{hypApriori4} hold,  $a, b, c, d, h_1, h_2$ are H\"older continuous
 and for some $\epsilon>0$
\begin{equation}
\label{hypExist1}
\inf_{x\in\Omega,\ u, v>0} u^{-1}\,h_1(x,u,v)>-\infty, \qquad
\inf_{x\in\Omega,\ u, v>0} v^{-1}\,h_2(x,u,v)>-\infty,
\end{equation}
\begin{equation}
\label{hypExist2}
\sup_{x\in\Omega,\ u>0} u^{-1}\,h_1(x,u,0)< \lambda_1^+(-\mathcal{M}^+, \Omega),\qquad
\sup_{x\in\Omega,\ v>0} v^{-1}\,h_2(x,0,v)< \lambda_1^+(-\mathcal{M}^+, \Omega),
\end{equation}
\begin{equation}
\label{hypExist3}
\sup_{x\in\Omega,\ u,v\in (0, \epsilon)^2} (\max\{u,v\})^{-1}\,h_i(x,u,v)<
 \lambda_1^+(-\mathcal{M}^+,\Omega),\qquad i=1,2.
\end{equation}
where $\mathcal{M}^+$ is the Pucci maximal operator (see Definition \ref{def31} below).

Then there exists  a bounded positive classical solution of (\ref{system_pqr_Dir2}).
\end{thm}
\begin{rmq}
If $\F$ is a HJB operator we can replace the first eigenvalue of the Pucci operator by the first eigenvalue of $\F$ in the above theorem.
\end{rmq}

   \section{Preliminaries}\label{sect3}
   For the convenience of the reader, we begin by recalling some definitions.
\begin{Def}\label{def31}
Let $0<\lambda<\Lambda$.  The extremal Pucci operators are defined by
$$\mathcal{M}^+(M)=\Lambda \underset{\mu_i>0}{\sum} \mu_i+\lambda \underset{\mu_i<0}{\sum} \mu_i = \sup_{\lambda I \le A \le \Lambda I} \mathrm{tr}(AM),\qquad\mathcal{M}^-(M)= -\mathcal{M}^+(-M),$$
for any symmetric matrix  $M\in \mathcal{S}_n$, where  $(\mu_i)_{i=1..n}$ are the eigenvalues of $M$.
\end{Def}

\begin{Def}\label{def32}
 $\F$ is an Isaacs operator if the following conditions are satisfied  :
\begin{itemize}
\item $\F$ is uniformly elliptic and Lipschitz:
there exist $\Lambda>\lambda>0$, $B\ge0$, such that for all symmetric matrices  $M,N$, and all $p,q\in \rn$, $x\in\Omega$,
   $$  \mathcal{M}^-(M-N) - B|p-q|\leq F(M,p,x)-F(N,q,x)\leq \mathcal{M}^+(M-N) + B|p-q|, \leqno(H_1)$$
\item $\F$ is 1-homogeneous: for all $t\geq 0$ and $M\in \mathcal{S}_n$, $p\in \rn$, $x\in\Omega$, we have
   $$  F(tM,tp,x)=t\,F(M,p,x). \leqno(H_2)$$
\end{itemize}
\end{Def}

We will also need the definition of the principal half-eigenvalue of an Isaacs operator. See  \cite{Armstrong} for more details.

\begin{Def}\label{def33}
Let $\Omega$ be a bounded domain of $\rn$ and $\F$ be an Isaacs operator.
We define the finite real number
$$\lambda_1^+(-\F,\Omega)=\sup\{\mu\in \mathbb{R}, \;\exists u\in C(\Omega),\;u>0,\; -\F[u]\geq \mu u\; \text{ in } \Omega\} .$$
\end{Def}

We recall all equalities and inequalities in this paper are understood in the viscosity sense. For the notion of viscosity solution, we refer the reader to \cite{CIL}, \cite{CafCab}.

We recall a transitivity result, whose proof  is a simple consequence of $(H_1)$ above and Lemma 3.2 in \cite{Armstrong}.

\begin{lem}
\label{lem_u-v}
Let $\F$ be an Isaacs operator.
Assume that  $f,g\in C(\Omega)$, $u,v\in C({\Omega})$ are viscosity solutions in $\Omega$  of
$$
\left\{ {\alignedat2
\F[u]&\geq f,\\
\F[v]&\leq g.
\endalignedat}\right.
$$
Then $w=u-v$ is a viscosity solution in $\Omega$  of
$$\mathcal{M}^+(D^2w) + B|Dw|\geq f-g. $$
\end{lem}

The next simple lemma is helpful in exploiting condition (\ref{condition_fg})
on the nonlinearities $f$ and $g$ of the system, as shown in the subsequent result.
\begin{lem}
\label{lem_|w|_|h|}
Assume that $\F$ satisfies $(H_1)$ and $\F[0]=F(0,0,x)=0$.
If  $w,h$ are continuous functions such that sign$(w)$=sign$(h)$ and $w$ is a viscosity solution in $\Omega$ of
 $$
 \left\{ {\alignedat2
 \F[w]&\geq h\\
 \F[-w]&\geq -h ,
 \endalignedat}\right.
$$
 then $|w|$
 is a viscosity solution in $\Omega$ of
 $$\F[|w|]\geq |h|$$
\end{lem}

\begin{proof}
Assume $\phi\in C^2(\Omega)$ touches by above $|w|$ at $x_0\in \Omega$.
If $w(x_0)>0$, then $h(x_0)\geq 0$ and since $\phi$ touches $w$ by above at $x_0$, we have
 $$F(D^2\phi(x_0), D\phi(x_0),x_0)\geq h(x_0)=|h(x_0)|.$$
 If $w(x_0)<0$, then $h(x_0)\leq 0$ and since $\phi$ touches $-w$ by above at $x_0$, we have
 $$F(D^2\phi(x_0), D\phi(x_0),x_0)\geq -h(x_0)=|h(x_0)|.$$
 If $w(x_0)=0$, then $h(x_0)=0$. Moreover, $\phi(x_0)=0$ and $\phi\geq |w|\geq 0$ so $x_0\in \Omega$  is a minimum point of $\phi$. Hence
$D^2\phi(x_0)\geq 0$, $D\phi(x_0)=0$, from which we deduce  $F(D^2\phi(x_0), D\phi(x_0),x_0)\geq F(0,0,x_0)=0=|h(x_0)|$.\end{proof}

\begin{lem}
\label{lem_|u-Kv|}Assume that $\F$ is an Isaacs operator.
Let $(u,v)$ be a viscosity solution of (\ref{mainsyst}) in $\Omega$ and assume that the nonlinearities $f,g\in C(\Omega)$  satisfy (\ref{condition_fg}).
Then
$$\mathcal{M}^+(D^2|u-Kv|)+B|D|u-Kv||\geq |f-Kg| \text{ \qquad in } \Omega $$
in the viscosity sense.
\end{lem}

\begin{proof}
Apply Lemma \ref{lem_|w|_|h|} to
$$w=u-Kv,\qquad\mbox{and}\qquad h=Kg(\cdot,u,v)-f(\cdot,u,v).$$
Observe that condition (\ref{condition_fg}) means that
$ h\,w\geq 0$.
Also, by the continuity of $f,g$ and (\ref{condition_fg}), it is easy to see that for all $x\in \Omega$ and $v\in \mathbb{R}$, we have
$Kg(x,Kv,v)-f(x,Kv,v)=0$,  hence if $w=0$ then $h=0$.
\end{proof}

The last lemma will be useful   when considering  system (\ref{system_Spqr}) on the whole space
(since   the auxiliary  function $Z=\min(u,Kv)$ will play a crucial role in our analysis).
\begin{lem}
\label{lem_min_uv} Assume that $\F$ is an Isaacs operator.
Let $\Omega$ be an open set and let $u,v,f,g,h\in C(\Omega)$.
Assume that
$u$ and $v$ are respectively viscosity solutions of
$$-\F[u]\geq f
\text{ \qquad and \qquad }-\F[v]\geq g $$
and that
\begin{equation}
 \label{condition_fgh}
 \left\{ {\alignedat2
 f\geq h & \text{\qquad on } \{u\leq v\}\\
 g\geq h & \text{\qquad  on } \{u>v\}.
 \endalignedat}\right.
 \end{equation}
Then
$ w:=\min(u,v)\in C(\Omega)$
is a viscosity solution in $\Omega$ of
$$-\F[w]\geq h. $$
\end{lem}

\begin{proof}
Assume $\phi\in C^2(\Omega)$ touches by below $w$ at $x_0\in \Omega$. Then $\phi \leq u$ and $\phi \leq v$.
If $u(x_0)\leq v(x_0)$ then   $w(x_0)=u(x_0)$ and $\phi$ touches $u$
by below  at $x_0$, hence
$$-F(D^2\phi(x_0), D\phi(x_0),x_0)\geq f(x_0)\geq h(x_0).$$
Similarly, if $u(x_0)>v(x_0)$ then   $w(x_0)=v(x_0)$ and $\phi$ touches $v$
by below  at $x_0$, hence
$$-F(D^2\phi(x_0), D\phi(x_0),x_0)\geq g(x_0)\geq h(x_0).$$\end{proof}

For the reader's convenience, we next recall  some known facts to which we
 refer in the subsequent proofs.

\begin{lem}
 \label{lem_harmonic_bounded}
Let  $F(D^2u)$ be an Isaacs operator.
If $u$ is a  viscosity solution of $F(D^2u)=0$ on $\rn$ and $u$ is bounded from below, then $u$ is constant.
\end{lem}

\begin{proof} The proof is well known and is the same as for the Laplacian. The result is also included in Theorem 1.7 in \cite{ASS2}.
\end{proof}

\begin{lem}
 \label{lem_nonexistence_superharmonic_if_alpha*<=0}
Let  $F(D^2u)$ be an Isaacs operator. Let $z$ a viscosity solution on $\rn$ of
$$ -F(D^2 z)\geq 0.$$
\begin{itemize}
\item[i)] Assume $\alpha ^*(F)\leq 0$.
If $z$ is bounded by below, then
$z$ is constant.
\item[ii)] Assume $\alpha ^*(F)>0$. Then for some $m>0$
$$z\geq \frac{m}{|x|^{\alpha ^*(F)}}\quad \text{ for all }\quad |x|\geq 1. $$
\end{itemize}

\end{lem}

\begin{proof}
The statement in i)  follows from Theorem 4.3 in \cite{armsir}, whereas
ii) is a particular case of Lemma 3.8 in \cite{armsir}. \end{proof}

\begin{lem}
\label{QSMP} Let $\Omega$ be an domain of $\rn$ and
let $F(D^2u )$ be an Isaacs operator.
Let $h\ge 0$ and $u\geq 0$ be a viscosity solution in $\Omega$ of
$$ -F(D^2u )\geq h.$$
Then, for any compact $K\subset \Omega$, there exist $\gamma=\gamma(\lambda,\Lambda,n)>0$,
$c=c(\lambda,\Lambda,n, K. \Omega)>0$ such that
$$\underset{K}{\inf}\; u\geq c \left(\int_{K }h^\gamma\right)^{\frac{1}{\gamma}} .$$

\end{lem}

\begin{proof}
This is proved in \cite{K}. The result in that paper is stated for a linear operator, but the same proof applies to any uniformly elliptic Isaacs operator.\end{proof}

It is well known (see for instance \cite{CIL}, \cite[Proposition~2.9]{CafCab}, \cite[Theorem 3.8]{CCKS}) that it is easy to pass to uniform limits with viscosity solutions.

\begin{lem}
 \label{lem_convergence_suite_de_solutions}
Let $\Omega$ be a domain of $\rn$ and $\F$ be an Isaacs operator.
Let $(u_j)$ a sequence of viscosity solutions
of
$$-\F[u_j]=f_j(x) \text{\qquad in }\Omega $$
where $f_j\in C(\Omega)$.
Assume that $u_j\rightarrow u$ and $f_j\rightarrow f$ locally uniformly in $\Omega$.
Then $u$ is a viscosity solution of
$$-\F[u]=f(x) \text{\qquad in }\Omega .$$
\end{lem}

We also recall the following solvability result.

\begin{lem}
\label{lem_existence_Dirichlet}
Let $\F$ be an Isaacs operator, $R>0$ and $f\in C(\overline{B_R})$.
Then there exists a  viscosity solution $u\in C(\overline{B_R})$ of the problem
\begin{equation}
\label{problem_dirichlet}
\left\{\;{\alignedat2
 -\F[u] &=f(x)     &\qquad& x\in B_R,\\
 u &= 0&\qquad& x\in\partial B_R
  \endalignedat}\qquad\right.
\end{equation}
as well as a unique viscosity solution $u\in C^{2,\alpha}(\overline{B_R})$ of the problem
\begin{equation}
\label{problem_dirichlet2}
\left\{\;{\alignedat2
 -\mathcal{M}^+(D^2 u) +B|Du|&=f(x)     &\qquad& x\in B_R,\\
 u &= 0&\qquad& x\in\partial B_R
  \endalignedat}\qquad\right.
\end{equation}
\end{lem}

\begin{proof} This result is contained in \cite{CCKS}, \cite{CKLS}.
\end{proof}

   \section{Liouville theorems for weighted inequalities on $\rn$}\label{sect4}
 In this subsection, we present two Liouville type results for inequalities on $\rn$. The first one concerns a coercive inequality and extends Lemma 3.4 from \cite{MSS}. That lemma for the Laplacian was proved in \cite{MSS} by using spherical means -- a tool obviously not available for more general operators. Here we give a different and more general proof, based on estimates from the regularity theory of elliptic equations, such as growth lemmas and Harnack type inequalities.
\begin{lem}
\label{lem_type_Lin}
Let $F(D^2u)$ be an Isaacs operator and $w\geq 0$ be a viscosity solution of
\begin{equation}
\label{equ_type_Lin}
F(D^2w)\geq \frac{A}{1+|x|^2}\;w^p\qquad\mbox{on }\;\rn,
\end{equation}
where $p> 0$ and $A>0$.
\begin{itemize}
\item[i)] If $w\not\equiv 0$, then $w$ is unbounded.
\item[ii)] If $p>1$, then $w\equiv 0$.
\end{itemize}
\end{lem}

\begin{proof}
By ($H_1$) the function $w$ is a viscosity solution
on $\mathbb{R}^n$ of
$$\mathcal{M}^+(D^2w)\geq \frac{A}{1+|x|^2}\;w^p.$$

\textit{Proof of i)}.
Assume $w\neq 0$.
We define $M(R)=\underset{B_{R}}{\sup}\;w$, and will show the existence of $c>0$ and $R_0>0$ such that for all $R\geq R_0$
$$M(2R)\geq M(R)+c ,$$
 which implies the statement of i).

For all $x\in \mathbb{R}^n$, we define
$$f(x)=\frac{A}{1+|x|^2}\;w(x)^p .$$
Since $f\geq 0$ and $f\in C(\rn)$, thanks to Lemma \ref{lem_existence_Dirichlet}, there exists a unique viscosity solution $u_R>0$ of
$$
\begin{array}{rclcc}
-\mathcal{M}^-(D^2 u_R)&=&f & \text{ on } &B_{2R}\\
u_R&=&0& \text{ on }  &\partial B_{2R}.
\end{array}
$$
We define on $\overline{B_{2R}}$ the function
$v_R=M(2R)-u_R$, which satisfies
$$
\begin{array}{rclcc}
\mathcal{M}^+(D^2 v_R)&=&f & \text{ on } &B_{2R}\\
v_R=M(2R)&\geq&w& \text{ on }&\partial B_{2R}.
\end{array}
$$
By the comparison principle we obtain
$w\leq v_R$ on $B_{2R}$,
which implies
$$\underset{B_R}{\inf} \,u_R+M(R) \leq M(2R).$$
Now, we define on $\overline{B_{2}}$ the function
$\tilde{u}_R(x)=u_R(Rx), $
which is a viscosity solution of
\begin{equation}
\label{equ_tilde_u_R}
-\mathcal{M}^-(D^2 \tilde{u}_R)=\frac{A\; R^2}{1+R^2|x|^2}\;w(Rx)^p\geq \epsilon \; w(Rx)^p  \text{\qquad  on } B_{2},
\end{equation}
for some $\epsilon>0$.
Since $w\neq 0$,  there exists $R_0>0$ such that
$ \underset{B_{1}}{\sup}\; w(R_0\cdot)> 0.$

Since $\mathcal{M}^+(D^2 w)\geq 0$,  by the local maximum principle applied
to $w(R\cdot)\geq 0$ for $R>0$ (see \cite[Theorem~4.8~(2)]{CafCab}),
for any $q>0$ there exists $C=C(q)>0$ such that for all $R\geq R_0$,
$$\|w(R\cdot)\|_{L^q(B_{2})} \geq C\, \underset{B_1}{\sup}\; w(R\cdot)
\geq C\, \underset{B_{1}}{\sup}\; w(R_0\cdot) =C>0 .$$
However $\tilde{u}_R\geq 0$ satisfies $(\ref{equ_tilde_u_R})$, so
by the quantitative strong maximum principle (see Lemma \ref{QSMP}) applied to \re{equ_tilde_u_R}, there exist $q_0>0$ and $c_0>0$ such that
$$\underset{B_1}{\inf}\; \tilde{u}_R \geq c_0\,\epsilon \;\|w(R\cdot)^p \|_{L^{q_0}(B_{1})} . $$
We choose $q=q_0\,p$ and obtain a constant $c>0$ such that for all $R\geq R_0$ we have
$$\underset{B_1}{\inf}\; \tilde{u}_R \geq c.$$
Since $\underset{B_1}{\inf}\; \tilde{u}_R=\underset{B_R}{\inf}\; u_R  $, this implies
$$M(2R)\geq M(R)+c.$$

\textit{Proof of ii)}. We will show that $w$ is bounded on $\mathbb{R}^n$, which implies $w=0$, by i). This can be proved similarly to \cite{MSS}, we include the full proof for completeness.

  As in \cite{Osserman}, we define the function $w_R\in C^2(B_R)$ by
 $$w_R(x)=C \frac{R^{2\alpha}}{(R^2-|x|^2)^\alpha} $$
 for all $x\in B_R$, where $\alpha=\frac{2}{p-1}.$
 It is easy to see, by direct computation, that if $C>0$ is
large enough then  for all $x\in B_R$
 \begin{equation}
\label{equation_w_R}
\Lambda\; \Delta w_R \leq \frac{A}{1+|x|^2}{w_R}^p
 \end{equation}
 in the classical sense. See for instance the computation on page 15 in \cite{MSS}.

Note $w_R$ is radial and its first and second radial derivatives  are nonnegative. Then $w_R$ is convex, so
 $\mathcal{M}^+(D^2w_R)=\Lambda\; \Delta w_R $.  Hence
\begin{equation}
 \label{equ_w_R}
 \mathcal{M}^+(D^2w_R)\leq \frac{A}{1+|x|^2}{w_R}^p.
 \end{equation}
Since $w_R(x)\underset{x\rightarrow \partial B_R}{\longrightarrow} +\infty$,  there exists $R'<R$ such that $w_R\geq \|w \|_{L^\infty(B_R)} $ on $ B_{R}\setminus B_{R'}$.\\
Assume that $\underset{B_{R'}}{\sup}\; [w-w_R] >0$.
 Since $w\leq w_R$ on $\partial B_{R'}$,  this supremum is attained at some $x_0\in B_{R'}$. On the other hand $w_R\in C^2(B_{R'})$, so $w_R$ is a legitimate test function for \re{equ_type_Lin}, and by the definition of a viscosity subsolution we get
 $$\mathcal{M}^+(D^2w_R(x_0)) \geq  \frac{A}{1+|x_0|^2}\;w(x_0)^p
 > \frac{A}{1+|x_0|^2}\;w_R(x_0)^p,$$
a contradiction with (\ref{equ_w_R}).
 This implies that $w\leq w_R$ on $B_{R'}$ and then on $B_{R}$.

 Now, for any  $x\in \mathbb{R}^n$, we let $R\to +\infty$, and obtain $w(x)\leq \lim_{R\to\infty} w_R(x)=C$. Hence $w$ is bounded on $\mathbb{R}^n$.
\end{proof}

\begin{lem}
 \label{lem_-F(D^2u)geq_V_u^r}
 Assume that $F(D^2u)$ is an Isaacs operator.
Let   $V\in C(\rn)$, $V\gneqq  0$ be such that for all
$\gamma>0$
$$\underset{R\rightarrow +\infty }{\liminf}\;
  \frac{1}{R^n}\int_{B_{2R}\setminus B_R} V^\gamma >0. $$
Assume that
\begin{equation}
 \alpha ^*(F)\le 0\qquad\mbox{or}\qquad 0\leq r\leq 1+\frac{2}{\alpha^*(F)} .
\end{equation}
Let $z\geq 0$ be a viscosity solution of
$$ -F(D^2 z)\geq V(x) \,z^r, \text{ \qquad  } x\in \rn.$$
Then
$z=0.$
\end{lem}

\begin{proof}
The case $\alpha ^*(F)\leq 0$ is obvious since the only viscosity solutions of
$ -F(D^2 z)\geq 0$
are the constants  (see
Lemma \ref{lem_nonexistence_superharmonic_if_alpha*<=0} i)) and because $V\geq 0,\;V\not\equiv 0$.
Hence we can suppose that
$\alpha ^*(F)>0$.

Assume for contradiction that $z\not\equiv0$. Since $z\geq 0$ and $-F(D^2z)\geq 0$,  by the strong maximum principle (see \cite[Proposition 4.9]{CafCab}) we have $z>0$ on $\rn$.

First, we note that the hypothesis on $V$ implies, for each $\gamma>0$, the existence of $R_0=R_0(\gamma)>0$ and $c_0=c_0(\gamma)>0$ such
that for all $R\geq R_0$,
$$ \left(\int_{B_2\setminus B_1}V(Rx)^\gamma dx\right)^{\frac{1}{\gamma}}\geq c_0.$$

Let $\gamma=\gamma(F,n)>0$ be the constant given in the quantitative strong maximum principle (see Lemma \ref{QSMP}) and $R_0=R_0(\gamma)$. Let $R\geq R_0$. We set
$ z_R(x)=z(Rx)$
and
$$m(R)=\underset{ B_R}{\inf}\; z=\underset{ B_1}{\inf}\; z_R.$$
It is easy to see that $z_R$ is a viscosity solution of
$$-F(D^2 z_R)\geq R^2V(Rx) {(z_R)}^r, \text{\qquad  }x\in \rn. $$
By the quantitative strong maximum principle (see Lemma \ref{QSMP}), there exists  $C>0$
 such that
$$\underset{B_1}{\inf}\; z_R \geq C \, \;\|R^2V(R\cdot) {z_R}^r \|_{L^{\gamma}(B_{1})} \geq C R^2 m(R)^r
\left(\int_{B_1}V(Rx)^\gamma dx\right) ^{\frac{1}{\gamma}}.$$
Hence, for all $R\geq R_0$
$$m(R)\geq C R^2  c_0(\gamma)\; m(R)^r. $$

From this point on the proof is similar to the one given in \cite{armsir} and in \cite{MSS} for the case when the elliptic operator is the Laplacian. Thus we only give a sketch of the proof, with more details in the third case below, where some differences appear.

\textbf{First case}. Assume $0\leq r\leq 1$.
For all $R\geq R_0$, since $m(R)>0$, we have
$$m(R_0)^{1-r}\geq m(R)^{1-r} \geq C R^2  c_0(\gamma) $$
because $R\mapsto m(R)$ is nonincreasing.
We then obtain a contradiction when  $R\to\infty$.

\textbf{Second case}. Assume $1<r<\frac{\alpha ^*(F)+2}{\alpha^*(F)}$, which is equivalent to
$\frac{2}{r-1} > \alpha ^*(F)$.
From the argument used in the previous case, we deduce that for all $R\geq R_0$
$$ m(R) \leq CR^{-\frac{2}{r-1}}, $$
for some $C>0$.  On the other hand,  for any $R\geq 1$,
$$ m(R)\geq {c}{R^{-\alpha ^*(F)}},$$
for some $c>0$ (this follows from Lemma \ref{lem_nonexistence_superharmonic_if_alpha*<=0} ii)). The last two inequalities  yield a contradiction, by letting $R\to+\infty$.

 \textbf{Third case}. Assume $r=\frac{\alpha ^*(F)+2}{\alpha^*(F)}$, which is equivalent to
$\frac{2}{r-1} = \alpha ^*(F) .$
As a consequence of this equality, if we set
 $$\tilde{z}_R=R^{\alpha^*(F)}z_R, \qquad \mbox{and}\qquad
\tilde{m}(R):=\underset{\partial B_R}{\inf}  \;\frac{z}{\Phi}
=\underset{\partial B_1}{\inf}  \;\frac{\tilde{z}_R}{\Phi},$$
and we can check that
 \begin{equation}
 \label{encadrement_m_tilde_R}
\frac{c}{M_1}\leq \tilde{m}(R)\leq  C.
\end{equation}

It is easy to see that $\tilde{z}_R$ is a viscosity solution of
$$ -F(D^2\tilde{z}_R)\geq V(Rx){\tilde{z}_R}^r, \text{\qquad }x\in \rn.$$
Hence $-F(D^2\tilde{z}_R)\geq 0$ so, by Lemma \ref{lem_nonexistence_superharmonic_if_alpha*<=0} ii), we deduce that
 \begin{equation}
 \label{minoration_z_bar_R}
 {z}_R- \tilde{m}(R)\Phi\geq 0 \text{\qquad on } \rn\setminus B_1.
 \end{equation}
Since
$ F(D^2(\tilde{m}(R)\Phi))=\tilde{m}(R)\;F(D^2\Phi)=0$,
and
$$ F(D^2\tilde{z}_R)\leq -V(Rx){\tilde{z}_R}^r,$$
by applying Lemma \ref{lem_u-v} we get
$$-\Mm\left(D^2\left[\tilde{z}_R-\tilde{m}(R)\Phi \right]\right)=\Mp\left(D^2 \left[\tilde{m}(R)\Phi-\tilde{z}_R\right]\right)\geq V(Rx){\tilde{z}_R}^r. $$
We apply the quantitative strong maximum principle (Lemma \ref{QSMP}) to the operator $\Mm$ with $\Omega=B_5\setminus \overline{B_1}$ and $K=\overline{B_4}\setminus B_2$. We find $\gamma^->0$ and
$c_->0$ such that
$$ \underset{\overline{B_4}\setminus B_2}{\inf}\; \left[\tilde{z}_R-\tilde{m}(R)\Phi\right]\geq c_-
\left(\int_{\overline{B_4}\setminus B_2}V(Rx)^{\gamma^-} {\tilde{z}_R(x)}^{r\gamma^-}\;dx\right)^{\frac{1}{\gamma^-}}.$$
By the $[-\alpha^*(F)]$-homogeneity of $\Phi$  we have
$ \Phi\geq c_1$  on $ \overline{B_4}\setminus B_2,$
which implies, by (\ref{encadrement_m_tilde_R}) and (\ref{minoration_z_bar_R}), that
$$\tilde{z}_R\geq {c_2} \text{\qquad on } \overline{B_4}\setminus B_2.$$
Hence we obtain
$$ \underset{\overline{B_4}\setminus B_2}{\inf}\; \left[\tilde{z}_R-\tilde{m}(R)\Phi\right]
\geq c_3\left(\int_{\overline{B_4}\setminus B_2 }V(Rx)^{\gamma^-}\;dx\right)^{\frac{1}{\gamma^-}}\geq c_4>0$$
for $R$ large enough, from the hypothesis on $V$.
Therefore, for $R$ large enough and
for all $|x|=1$ we have
$$\tilde{z}_R(2x)\geq \tilde{m}(R)\Phi(2x)+c_1, $$
from which we deduce
$$ \tilde{m}(2R)\geq \tilde{m}(R)+\frac{c_1 2^{\alpha^*(F)}}{M_1},$$
so $\tilde{m}(R)$ is unbounded, a contradiction.\end{proof}

 The previous theorem can obviously be applied to any constant function $V$, but also for any nontrivial nonnegative subsolution, as we show in the
  following lemma.
 \begin{lem}
\label{lem_liminf_V_sousharmonique}
 Let $F(D^2u)$ be an Isaacs operator.
Let $V\in C(\rn)$, $V\gneqq 0$ be such that
$$F(D^2 V)\geq 0\qquad\mbox{on }\;\rn. $$
Then, for any $\gamma >0$,
$$\underset{R\rightarrow +\infty }{\liminf} \frac{1}{R^n}\int_{B_{2R}\setminus B_R} V^\gamma >0. $$
\end{lem}

\begin{proof}
 This is a consequence of \cite[Theorem 4.8 (2)]{CafCab}.
Indeed, if
we apply the latter to
$V_R=V(R\cdot)\geq 0$ for $R>0$, then  for any $\gamma>0$
there exists $C=C(\gamma)>0$ such that for all $R>0$,
$$\left( \frac{1}{R^n}\int_{B_{2R}\setminus B_R}
V^\gamma\right)^{\frac{1}{\gamma}}= \left( \int_{B_{2}\setminus B_1} V(Rx)^\gamma
\,dx\right)^{\frac{1}{\gamma}}
\geq  C \underset{B_{\frac{5}{3}R}\setminus B_{\frac{4}{3}R}}{\sup}V
\geq C \underset{\partial B_{\frac{5}{3}R}}{\sup}V
= C \underset{ B_{\frac{5}{3}R}}{\sup}V, $$
where the last equality follows from the comparison principle and $F(D^2
V)\geq 0. $
But since $V\not\equiv 0$, there exists $R_0>0$ such that
$\underset{ B_{\frac{5}{3}R_0}}{\sup}V>0. $\end{proof}

\section{Proportionality results for systems on $\rn$ or a cone}\label{sect5}

\subsection{Case of a cone}

In this subsection, we fix a cone
$ \co=\{tx, t>0, x\in \omega\}$,
 where
 $\omega$ is  a $C^2$-smooth subdomain of the unit sphere $S_1$.
We will use the following notation $$B_R^+=B_R\cap \co,\qquad S_R^+=S_R\cap \co,$$
and, given an Isaacs operator $F(D^2u)$, we will denote by $\Psi^\pm\in C\left(\overline{\co}\setminus \{0\}\right)$  the positive solutions  of
\begin{equation}
\label{}
\left\{\ {\alignedat2
 -F(D^2 \Psi^\pm)&=0,&\qquad \text{ in } \co\\
 \Psi^+&=0,&\qquad \text{ on  } \partial \co \setminus \{0\}
 \endalignedat}\right.
 \end{equation}
normalized so that $\Psi^\pm(x_0)=1$ for some given point $x_0\in\co$, and such that
$$\Psi^\pm(x)>0 \text{  on $\co$ \qquad and \qquad } \Psi^\pm(x)=t^{\alpha^\pm}\Psi^\pm(tx) \text{\qquad for all }t>0,\; x\in \co.$$
Here $\alpha^+=\alpha^+(F,\co)>0>\alpha^-=\alpha^-(F,\co)$
are uniquely determined.
For more details on these functions and their construction, see \cite{ASS}. Recall that when $F$ is the Laplacian and $\co$ is the half-space $\{x_n>0\}$, we have $\Psi^+(x) = x_n|x|^{-n}$, and  $\Psi^-(x) = x_n$.

We next recall the following Phragm\`en-Lindelh\"of principle for fully nonlinear equations, which is
 a particular case of the results in \cite{CDV} or \cite{ASS}.

\begin{lem}
\label{lem_phragmen_lindelhof}
Assume that $\F$ is an Isaacs operator.
Let $w\in C(\overline{\co})$ be a bounded viscosity solution  of
\begin{equation}
\label{equ_phragmen_1}
\F[w]\geq 0 \text{\qquad on }\co
\end{equation}
\begin{equation}
\label{equ_phragmen_2}
w\leq 0 \text{\qquad on }\partial \co.
\end{equation}
 Then
$ w\leq 0$ in $\co$.
\end{lem}

\begin{proof}
This is a special case of \cite[Theorems~A]{CDV} or \cite[Theorems~1.6-1.7]{ASS} and their proofs. Using the notation of
\cite{ASS}, we set $\Omega=\Omega'=\co$, so
$\mathcal{D}=\co$ and we choose $\mathcal{D}'=\co$.
Since $w$ is bounded and $\alpha^-<0<\alpha^+$, then condition (1.12) of
\cite[Theorem 1.7]{ASS}
is clearly satisfied.
\end{proof}

\begin{proof}[Proof of Theorem \ref{thm_proportionnalite_hn}: ] By combining Lemma \ref{lem_|u-Kv|} and Lemma \ref{lem_phragmen_lindelhof} we get $|u-Kv|\le 0$ in $\co$, i.e. $u\equiv Kv.$
\end{proof}

We now turn to the proof of Theorem \ref{newhs}. The proof of the corresponding result for the Laplacian in \cite{MSS} depended heavily on the use of half-spherical means. Obviously, this tool cannot be used for more general operators, so a different proof is needed. Our proof here is shorter and uses the maximum principle and the boundary Harnack inequality.
\medskip

{\it Proof of Theorem \ref{newhs} (i)}.
We define the quotient
$$q_u(r) = \inf_{ S_r^+} \frac{u}{\Psi^-}.$$
Observe that the Hopf lemma applied to $-\F[u]\ge 0$ implies $\frac{\partial u}{\partial \nu}\ge c_r>0$ on $S_r^+\cap\partial \co$, while the boundary Lipschitz estimates applied to $-\F[\Psi^-]= 0$ imply $\frac{\partial \Psi^-}{\partial \nu}\le C_r$ on $S_r^+\cap\partial \co$ (here $\nu$ denotes the interior normal to $\partial \co$), therefore $q_u(r)>0$.

Since $\Psi^-=0$ on $\partial \co$ and $u\ge0$ we have  $u\ge q(r) \Psi^-$ on the boundary $\partial B_r^+$. By the comparison principle we have this inequality in $B_r^+$. In other words, $q_u(r) = \inf_{B_r^+} \frac{u}{\Psi^-}$. Therefore the function $q_u$ is decreasing in $r$, and hence has a limit as $r\to \infty$, which we denote with $L_u\ge 0$. In particular, we have $u\ge L_u\Psi^-$ in  $\co$.

Similarly, we define $L_v\ge0$ such that $v\ge L_v\Psi^-$ in  $\co$.
Now, if both $L_u>0$ and $L_v>0$ we get from (\ref{mainsyst})
$$
-\F[u]\ge c(\Psi^-)^{\sigma}
$$
for some $\sigma>0$. This contradicts Theorem \ref{lioucone} since $\Psi^-\ge c|x|^{-\alpha^-}$ in any proper subcone of $\co$.

Let us assume $L_u=0$. We are going to prove that $u\le Kv$. Set $w=(u-Kv)_+$. Then we know that
$$
-\F[w]\le 0 \le -F[u]\quad \mbox{and}\quad w\le u\qquad\mbox{in }\;\co.
$$
Let $z_R$ be a solution of
$$
\left\{
\begin{array}{rclcc}
\F[z_R] &=&0&\mbox{in}& B_R^+\\
z_R&=& w&\mbox{on}&\partial B_R^+.
\end{array}
\right.
$$
By the global $C^{1,\alpha}$-estimates (see for instance Proposition 2.2 in \cite{ASS}) we see that the sequence $z_R$ is locally uniformly bounded as $R\to\infty$ in each fixed compact of $\overline{\co}$. Therefore we can pass to the limit and get a function $z$ such that
$$
w\le z\le u\quad \mbox{and}\quad F(D^2z) =0\qquad\mbox{in }\;\co.
$$

We now define
$$
Q_w(r) := \sup_{S_r^+} \frac{w}{\Psi^-},
$$
and we see that $Q_w(r)$ is increasing in $r$, by an argument similar to the one we used for $q_u$.

On the other hand
$$
Q_w(r) =  \sup_{B_r^+} \frac{w}{\Psi^-} \le \sup_{ B_r^+} \frac{z}{\Psi^-} \le C\inf_{ B_r^+} \frac{z}{\Psi^-} \le C\inf_{ B_r^+} \frac{u}{\Psi^-} = Cq_u(r),
$$
where we used the boundary Harnack inequality for the functions $z$ et $\Psi^-$ -- see for instance Proposition 2.1 in \cite{ASS}. Now $\lim_{r\to\infty} q_u(r)=L_u=0$ means the nonnegative increasing function $Q_w$ tends to zero as $r\to \infty$, that is, $Q_w\equiv0$.\smallskip

{\it Proof of Theorem \ref{newhs} (ii)}. To deduce (ii) from (i) we use exactly the same argument as in the proof of Theorem 2.8 in \cite{MSS}, replacing the reference to Lemma 3.1 there by a reference to Theorem~\ref{lioucone} above. Observe the conditions (2.6) and (2.7) in Theorem 2.8 in \cite{MSS} are exactly our \re{condition_rnsd}--\re{condition_rnsd2} with $\alpha^+=n-1$, $\alpha^-=-1$. \hfill $\Box$

\subsection{Case of  the whole space}

In this section, we focus on the system
\begin{equation}
\label{*system_Spqr}
\left\{\quad{\alignedat2
-F(D^2 u)&=u^rv^p[av^q-cu^q] &\text{\qquad on }\mathbb{R}^n\\
 -F(D^2 v)&=v^ru^p[bu^q-dv^q]&\text{\qquad on }\mathbb{R}^n.
 \endalignedat}\right.
\end{equation}
In our study of \eqref{*system_Spqr} we  always assume that the real parameters $a,b,c,d,p,q,r$ satisfy
\begin{equation}\label{Hyp_abcdpqr1}
a, b>0,\quad c,d\ge 0, \qquad p,r\ge 0,\quad q>0,\quad q\ge |p-r|.
\end{equation}

The last hypothesis provides the following result, proved in the appendix of  \cite{MSS}.
\begin{propo}
\label{Prop-K}
Assume (\ref{Hyp_abcdpqr1}).
\smallskip

\noindent (i) Then  the nonlinearities $f$ and $g$ in the system \eqref{*system_Spqr} satisfy \eqref{condition_fg} for some $K>0$. 
\smallskip

 \noindent (ii) Assume moreover that $ ab \geq cd$. Then the number $K>0$ is unique.
 We have
  $K=1$ if and only if $a+d=b+c$ and  $K>1$ if and only
if $a+d>b+c$. In addition, if $ab>cd$  (resp. $ab=cd$),
then $a-cK^q>0$  (resp. $=0$),  $bK^q-d>0$ (resp. $=0$).
\end{propo}

As in \cite{MSS}, in what follows we set
$$Z=\min(u,Kv),\qquad\mbox{and}\qquad
W=|u-Kv| $$
and we establish a system of elliptic inequalities satisfied by $Z$ and $W$.

\begin{lem}
\label{lem_principal_rn}
Let $F(D^2u)$ be an Isaacs operator.
Assume that (\ref{Hyp_abcdpqr}) holds and let $(u,v)$ be a positive
viscosity solution of (\ref{*system_Spqr}). Assume that $ ab \geq cd$.
\begin{itemize}
\item[a)] The functions $Z$ and $W$ are viscosity solutions on $\rn$ of
$$-F(D^2Z)\geq 0,\qquad \Mp(D^2 W)\geq 0 .$$
\item[b)] If $p+q< 1$, suppose in addition that $(u,v)$ is bounded.
Then $Z$ is a viscosity solution of
\begin{equation}
\label{DeltaZineq}
-F(D^2 Z) \ge CW^\beta Z^r\qquad\mbox{in }\; \rn,
\end{equation}
where $C>0$ and $\beta:=\max(p+q,1).$
\item[c)] Assume $r>p$ and $c, d>0$.
If $q+r< 1$, suppose in addition that $(u,v)$ is bounded.
Then $W$ is a viscosity solution of
\begin{equation}
\label{DeltaWineq}
\mathcal{M}^+(D^2 W) \ge CZ^p W^\gamma\qquad\mbox{in }\; \rn,
\end{equation}
where $C>0$ and $\gamma:=\max(q+r,1).$
\item[d)] Assume $ab>cd$. Then $Z$ is a viscosity solution of
$$-F(D^2 Z) \ge CZ^{p+q+r}\qquad\mbox{in }\; \rn.$$
\end{itemize}
\end{lem}

Most of the arguments in the proof of this lemma are  similar to \cite{MSS} (except in c)), thanks to Lemma \ref{lem_|u-Kv|} and Lemma \ref{lem_min_uv}. We sketch the proof here.

\begin{proof}
\textbf{ a)} By Proposition~\ref{Prop-K}, we have
\begin{equation}
\label{abcdK}
a\geq cK^q,\quad bK^q\geq d .
\end{equation}
Hence, on the set $\{u\le Kv\}$, we have
$$f(u,v)=u^rv^p[av^q-cu^q]\geq c u^rv^p[(Kv)^q-u^q]\geq 0$$
and similarly on the set $\{u > Kv\}$, we have
$g(u,v)\geq 0.
$
Now, we apply Lemma \ref{lem_|u-Kv|} and Lemma \ref{lem_min_uv} to $u$ and $Kv$ with $h=0$ and deduce a).

\textbf{ b)} We have
$x^q-y^q\ge C_q x^{q-1}(x-y)$, for any real $x\geq y\geq 0$,
where $C_q =1$ if   $q \geq 1$ and  $C_q =q$   if   $0<q<1$.
Using (\ref{abcdK}), on the set $\{u\le Kv\}$ we have
\begin{align*}
f(u,v)&=u^rv^p[av^q-cu^q]\geq \frac{a}{K^q}\,u^r v^p[(Kv)^q-u^q]\\
&\geq \frac{a C_q }{K}  u^rv^{p+q-1}(Kv-u)\geq \frac{a  C_q }{K^{p+q}} u^r(Kv)^{p+q-1}(Kv-u)
\geq C_1 Z^rW^\beta
\end{align*}
for some $C_1>0$ (we use that if $p+q\geq 1$, we have $(Kv)^{p+q-1}\geq (Kv-u)^{p+q-1}$,
whereas if $p+q<1$, then $(Kv)^{p+q-1}\geq C_0$ for some $C_0>0$, since $v$ is assumed bounded).
Similarly, on the set $\{u > Kv\}$, we have $g(u,v)\geq C_2 Z^rW^\beta$,  for some $C_2>0$.

Now, we apply Lemma \ref{lem_min_uv} to $u$ and $Kv$ with $h=C Z^rW^\beta$ where $C= \min(C_1,K C_2)$ (since $-F(D^2[Kv])=Kg(u,v)$) and  obtain that
$Z=\min(u,Kv)$  is a viscosity solution of
$$ -F(D^2Z)\geq C Z^rW^\beta.$$

\textbf{d)} Since $ab>cd$, we know from Proposition~\ref{Prop-K} that for some small $\epsilon>0$ we have $a\geq cK^q+\epsilon$, $bK^q\geq d+\epsilon$.
Hence,  on the set $\{u\le Kv\}$ we have
\begin{align*}
f(u,v)&=u^rv^p[av^q-cu^q]\geq \epsilon \,u^r v^{p+q}+ c\,u^r v^p[(Kv)^q-u^q]\\
&\geq \epsilon \,u^r v^{p+q}
\geq \epsilon K^{-p-q}\,Z^{r+p+q}.
\end{align*}
and, similarly, on the set $\{u> Kv\}$ we have the same inequality for $g$.

We again apply Lemma \ref{lem_min_uv} to $u$ and $Kv$ with $h=C Z^{r+p+q}$, where we set $C= \epsilon\min(K^{-p-q}, K^{1-q-r})$, and  obtain that
$Z=\min(u,Kv)$  is a viscosity solution of
$$ -F(D^2Z)\geq C Z^{p+q+r}.$$

\textbf{c)}
 Thanks to Lemma 7.1 i) in \cite{MSS}, we know that, since $r>p$ and $c, d>0$, we have for some $C_0>0$
$$|Kg(u,v)-f(u,v)|\geq C_0 \;u^pv^p (u+Kv)^{q+r-p-1}|u-Kv|. $$
We also note that for $x,y \geq 0$ and $x+y>0$, we have $$\frac{xy}{x+y}\geq \frac{1}{2}\min(x,y) .$$
Hence,
\begin{align*}
|Kg(u,v)-f(u,v)| &\geq  \frac{C_0}{K^p}\left[\frac{u\; Kv}{u+Kv}\right]^p(u+Kv)^{q+r-1}|u-Kv|\\
& \geq \frac{C_0}{(2K)^p}Z^p (u+Kv)^{q+r-1}|u-Kv|\geq C\,Z^p W^\gamma,
\end{align*}
for some $C>0$  (using that if $q+r\geq 1$, we have $(u+Kv)^{q+r-1}\geq |u-Kv|^{q+r-1}$,
whereas if $q+r<1$, we have $(u+Kv)^{q+r-1}\geq C'_0$ for some $C'_0>0$,  since $u$ and $v$ are assumed bounded).

Now, thanks to Lemma \ref{lem_|u-Kv|}, we obtain that in the viscosity sense
$$ \mathcal{M}^+ (D^2 W) \ge CZ^p W^\gamma\qquad\mbox{in }\; \rn.$$
\end{proof}

We now can give the following

\begin{proof}[Proof of Theorem \ref{thm_proportionnalite_rn}. ]
Let $(u,v)$ be a positive viscosity solution of (\ref{*system_Spqr}) in $\rn$. Let $Z$ and $W$ be defined as in the previous lemma.\smallskip

\textbf{Proof of i)} Assume that $W\neq 0$.
From Lemma \ref{lem_principal_rn}  b), we know that
 $Z$ is a viscosity solution of
\begin{equation}
-F(D^2 Z) \ge C\,V\, Z^r\qquad\mbox{in }\; \rn,
\end{equation}
where $C>0$ and
 $V=W^\beta$, with $\beta:=\max(p+q,1).$
From Lemma \ref{lem_principal_rn}  a), we know that
$W$ is a viscosity solution on $\rn$ of
$\Mp(D^2 W)\geq 0 $
so, since $\beta \geq 1$,
$V$ is a viscosity solution on $\rn$ of
$$\Mp(D^2 V)\geq 0 .$$
  Moreover, $V\gneqq 0$, hence,
by Lemma \ref{lem_liminf_V_sousharmonique}, $V$ satisfies the hypothesis of Lemma~\ref{lem_-F(D^2u)geq_V_u^r}. Therefore $Z=0$, a contradiction since $u,v>0$. Then $W=0$, i.e. $u=Kv$.\smallskip

\textbf{Proof of ii)} We can assume $$q+r>1.$$ Indeed, if $q+r\leq 1$, then
$u,v$ are bounded and  $r\leq q+r\leq 1<\frac{\alpha ^*(F)+2}{\alpha^*(F)}$ so the result follows from  a), which we already proved.

We can also assume that $$r>p.$$
Indeed, if $p+q< 1$, then $p+q<1<q+r$ so $r>p$;  if $p+q\geq 1$, we can assume $r> \frac{\alpha ^*(F)+2}{\alpha^*(F)}$ (or else the result follows from a)), so $r>\frac{2}{\alpha^*(F)}\geq p.$

Since $Z$ is a viscosity solution on $\rn$ of
$-F(D^2Z)\geq 0, $
 by Lemma \ref{lem_nonexistence_superharmonic_if_alpha*<=0} ii) there exists $m>0$ such that
$$Z^p\geq \frac{m^p}{|x|^{p\,\alpha ^*(F)}} \qquad \text{ for all } x\in\rn \setminus B_1.$$
Therefore since $ p\leq \frac{2}{\alpha ^*(F)}$,  there exists $A_1>0$ such that
$$Z^p\geq \frac{A_1}{1+|x|^2} \qquad \text{ for all }x\in\rn \setminus B_1.$$
Moreover, $Z^p>0$ is continuous on $B_2$ so there exists $A_2>0$ such that
$Z^p\geq \frac{A_2}{1+|x|^2}$ for all $x\in B_1$.

Therefore, since $r>p$ and $c,d>0$, by Lemma \ref{lem_principal_rn}  c)  $W$ is a viscosity solution of
$$\mathcal{M}^+(D^2 W) \ge \frac{A}{1+|x|^2}  W^\gamma \qquad \text{ in } \; \rn,$$
for some $A>0$ and $\gamma=q+r>1$.
Then, by Lemma \ref{lem_type_Lin}, we have $W=0.$\end{proof}

\begin{proof}[Proof of  Theorem \ref{thm_liouville_sol_positive_rn}.] Let $(u,v)$ be a bounded positive viscosity solution on $\rn$ of (\ref{*system_Spqr}). Then from
Theorem \ref{thm_proportionnalite_rn} we know that
$u=Kv.$
Hence $v$ is a solution of
$$-F(D^2v)=K^{p}(bK^q-d)v^\sigma .$$
But since $ab>cd$, we know from  Proposition \ref{Prop-K} that $bK^q-d>0$, so, by using the scaling of the equation and the hypothesis, we get $v=0$ and hence $u=0$.\end{proof}
\smallskip

\begin{proof}[Proof of Theorem \ref{thm_liouville_rn}].
Let $(u,v)$ be a nonnegative bounded viscosity solution of (\ref{*system_Spqr}).
First note that $u$ is a viscosity solution of
$$-F(D^2u)+Cu\geq 0$$
where $C=c\,v^p\,u^{q+r-1}\geq 0$.
We can apply the strong minimum principle and deduce that
$u=0$ or $u>0$. The same is true for $v$. By Theorem \ref{thm_liouville_sol_positive_rn} at least one of $u$, $v$ vanishes.

Assume for instance that $u=0$.
First, if $p>0$ then $-F(D^2v)=0$, so by Lemma~\ref{lem_harmonic_bounded} we have $v=C_2\geq 0$. Second,
if $p=0$ then $F(D^2v)=d\,v^{q+r}$.
 Now, if $d=0$, then $v=C_2$, while  if $d>0$,
then $v=0$ by Lemma \ref{lem_type_Lin} because $q+r>1$.

The same analysis can be done if $v=0$. Hence, in all cases $(u,v)$ is semitrivial.
Finally, if $r=0$, then it is clear that if $u=0$, then $v=0$, and vice versa. The last statements are obvious.\end{proof}

\section{A priori estimates and an existence result in a bounded domain}\label{sect6}

In this section we prove Theorem~\ref{AprioriBound2}.
We recall that we consider the following system with  general lower order terms
\begin{equation}
\label{system_pqr_Dir2}
\left\{\quad{\alignedat2
-\F[u]&=u^rv^p\bigl[a(x)v^q-c(x)u^q\bigr]+h_1(x,u,v), &\qquad&x\in \Omega,\\
-\F[v]&=v^ru^p\bigl[b(x)u^q-d(x)v^q\bigr]+h_2(x,u,v), &\qquad&x\in \Omega,\\
u&=v=0, &\qquad&x\in \partial\Omega,\\
 \endalignedat}\right.
\end{equation}
where $\Omega$ is a bounded Lipschitz  domain of $\mathbb{R}^n$, and $\F$ satisfies
 $(H_1)-(H_2)$.

\begin{proof}[\textit{Proof of Theorem~\ref{AprioriBound2}}.] The proof of this theorem follows the same scheme as the proof of Theorems 6.1 and 6.2 of \cite{MSS}, where $\F$ was assumed to be the Laplacian. We will not repeat the parts of the proof where it is similar to the one in the previous paper, and will only highlight the differences (most of which appear after the definition of the function $\mathcal{S}$ below).

We  consider the following parametrized version of system (\ref{system_pqr_Dir2}),
\begin{equation}
\label{system_pqr_Dir3}
\left\{\ {\alignedat2
-\F[u] &=F(t,x,u,v), &\qquad &x\in \Omega,\\
 -\F[v]&=G(t,x,u,v), &\qquad&x\in \Omega,\\
u&=v=0, &\qquad &x\in \partial\Omega,\\
 \endalignedat}\right.
\end{equation}
where $F$ and $G$ are defined as in \cite{MSS}.

We perform exactly the same ``blow-up" change of variables as in \cite{MSS}. The only difference is that now the modified functions $\tilde u_j, \tilde v_j$ satisfy a system where appears the elliptic operator $$F(D^2 \tilde u_j, \lambda_j D \tilde u_j , x_j+ \lambda_jy),$$
and $\lambda_j\to0$. By using the global $C^{1,\alpha}$-estimates for Isaacs operators and Ascoli's theorem, we can extract a subsequence of $(\tilde u_j, \tilde v_j)$ which converges in $C^1_{\mathrm{loc}}$ (the way to perform such a limit argument for fully nonlinear operators is described in extenso in \cite{QSsys}). After the passage to the limit by Lemma \ref{lem_convergence_suite_de_solutions}, we obtain a bounded nonnegative viscosity solution $(U,V)$ of
\begin{equation}
\label{proofAprioriUV}
\left\{\;{\alignedat2
 -F(D^2 U,0,x_0) &= U^r V^p\bigl[a_0V^q-c_0U^q\bigr],    \\
 -F(D^2 V,0,x_0) &= V^r U^p\bigl[b_0U^q-d_0V^q\bigr],    \\
  \endalignedat}\qquad\right.
\end{equation}
either in $\rn$ or in a cone of $\rn$ which resembles locally  the boundary of $\Omega$ at a neighborhood of some point of $\partial \Omega$.
In addition $c_0d_0<a_0b_0$  in view of (\ref{hypApriori2}), and $U(0)\geq 2^{-\alpha}>0$. By the strong maximum principle $U>0$ everywhere.

If the limiting domain is a cone, then Theorem \ref{thm_proportionnalite_hn} implies that $U$ and $V$ are proportional, hence both $U$ and $V$ are positive. This contradicts the assumption of  Theorem~\ref{AprioriBound2}.

So we can assume \re{proofAprioriUV} is set in $\rn$. By performing a rotation in $\rn$ we may assume the operator in the left-hand side of \re{proofAprioriUV} is $\overline{F}_{Q,x_0}$, for any initially chosen orthogonal matrix $Q$.

By the assumption of the theorem $U$ and $V$ cannot be both positive, so by  Theorem \ref{thm_liouville_rn} there is a constant $\bar C>0$ such that
$U\equiv \bar C$  and $V\equiv 0$. Hence
\begin{equation}
\label{proofApriori5}
\lim_{j\to\infty}(\tilde u_j, \tilde v_j)=(\bar C,0)\quad\hbox{ locally uniformly on $\mathbb{R}^n$.}
\end{equation}

Next, the exclusion of such semi-trivial rescaling limits is done exactly as in \cite{MSS} if we observe that the definition of the principal eigenvalue of a fully nonlinear operator (recall Definition \ref{def33}) trivially implies
$$\lambda_1^+(-\F,B_R)=\frac{\lambda_1^+(-\F,B_1)}{R^2}.$$

Thus the proof of the a priori bound \re{AprioriBoundM} is concluded.

The proof of the existence part of Theorem~\ref{AprioriBound2} goes again like in \cite{MSS}, replacing the Laplacian and its first eigenvalue by the fully nonlinear operator $\ofq$ and its first eigenvalue defined in Definition \ref{def33}, also observing that, by \cite{Armstrong} (see the discussion after Corollary 3.6 in that paper) that
\begin{equation}
\label{lambda_1_def_equivalente}
\lambda_1^+(-\mathcal{M}^+)\le \lambda_1^+(-\F)=\sup \{\lambda, \; -\F+\lambda \text{ satisfies the maximum principle}\}.
\end{equation}

We use the same fixed point theorem of  Krasnoselskii and Benjamin as in \cite{MSS} and the proof stays almost identical until the definition of the function $\mathcal{S}$. Instead of setting $\mathcal{S}=\sqrt{uv}$ as in \cite{MSS}, we now set
$$
\mathcal{S}=\min\{u,v\}.
$$
and we  prove that, in the viscosity sense,
\begin{equation}
\label{tuktu}
-\F[{\cal S}]\geq (A-C_1){\cal S} \quad \mbox{ in }\;\Omega.
\end{equation}
This implies that
$\lambda_1^+(-F,\Omega)\geq A-C_1$, which is in contradiction with the arbitrary choice of $A$ (see (6.12), (6.13), (6.24) in \cite{MSS} for more on the choice of $A$).

We know that, in the viscosity sense,
$$ \begin{aligned}
   -\F[u] &\geq  u^r v^p\bigl[(a(x)+A)v^q-c(x) u^q\bigr]+(A-C_1)u+A&\hbox{in }\Omega,\\
   -\F[v]&\geq v^r u^p\bigl[(b(x)+A)u^q- d(x) v^q\bigr]+(A-C_1)v+A& \hbox{ in }\Omega.
\end{aligned}
$$

Since we choose  $A$ such that $A>\max\{ \|c\|_\infty, \|d\|_\infty\}$, we have  $Av^q\ge c(x) u^q$ in the subset of $\Omega$ where $v\ge u$, and $Au^q\ge d(x) v^q$ in the subset of $\Omega$ where $u\ge v$. Therefore
$$ \begin{aligned}
   -\F[u] &\geq  (A-C_1)u =(A-C_1)\mathcal{S}&\hbox{in }\Omega\cap \{v\ge u\},\\
   -\F[v]&\geq (A-C_1)v=(A-C_1)\mathcal{S} & \hbox{ in }\Omega\cap \{u\ge v\}.
\end{aligned}
$$
By Lemma \ref{lem_min_uv} we get \re{tuktu}.

Finally, at the end of the proof, in order to show that the first hypothesis of the fixed point theorem (see Theorem 6.3 in \cite{MSS}) is  verified, we again argue by contradiction.
Now, for any (small) $\delta>0$ we can find a
 positive solution $(u,v)$ with $\|(u,v)\|\le \delta$,
of (\ref{system_pqr_Dir2}) with the right-hand side of this system
multiplied by {some} $\eta\in[0,1]$. By using \eqref{hypExist3} we obtain,
 with $\lambda_1=\lambda_1^+(-\F, \Omega)$ and for some $\epsilon_0>0$,
 $$ \begin{aligned}
   -\F[u] &\leq C u^rv^{p+q} +  (\lambda_1-\epsilon_0)\max\{u,v\}\le C(\max\{u,v\})^\sigma  +  (\lambda_1-\epsilon_0)\max\{u,v\},\\
   -\F[v]&\leq C v^r u^{p+q} +  (\lambda_1-\epsilon_0)\max\{u,v\}\le C(\max\{u,v\})^\sigma  +  (\lambda_1-\epsilon_0)\max\{u,v\}.
\end{aligned}
$$

As in Lemma 3.4 the maximum of subsolutions is a viscosity subsolution, so
$$
\begin{aligned}
-\F[\max\{u,v\}] &\le C(\max\{u,v\})^\sigma  +  (\lambda_1-\epsilon_0)\max\{u,v\}\\
&\le (C\delta^{\sigma-1}+\lambda_1-\epsilon_0)\max\{u,v\} \qquad\mbox{ in }\; \Omega
\end{aligned}
$$
Now, if we choose  $\delta$ sufficiently small,
 by \eqref{lambda_1_def_equivalente} we get $\max\{u,v\}\leq 0$, a contradiction.

Theorem~\ref{AprioriBound2} is proved.\end{proof}

\end{document}